\documentclass[11pt]{amsart}

\usepackage[a4paper,hmargin=3.5cm,vmargin=3.5cm]{geometry}
\usepackage{amsfonts,amssymb,amscd,amstext}
\usepackage{graphicx}
\usepackage[dvips]{epsfig}

\usepackage{fancyhdr}
\usepackage{times}

\usepackage{enumerate}
\usepackage{titlesec}
\usepackage{mathrsfs}

\def\Ccal{\mathcal{C}}
\def\Lcal{\mathcal{L}}
\def\Sscr{\mathscr{S}}
\def\Mscr{\mathscr{M}}
\def\Kscr{\mathscr{K}}
\def\Lscr{\mathscr{L}}

\def\Fscr{\mathscr{F}}
\def\M{\mathbb{M}}
\def\c{\mathbb{C}}
\def\d{\mathbb{D}}
\def\b{\mathbb{B}}
\def\h{\mathbb{H}}
\def\r{\mathbb{R}}
\def\n{\mathbb{N}}
\def\s{\mathbb{S}}
\def\mgot{\mathfrak{m}}

\titleformat{\section}
{\filcenter\bfseries\large} {\thesection{.}}{0.2cm}{}
\titleformat{\subsection}[runin]
{\bfseries} {\thesubsection{.}}{0.15cm}{}[.]
\titleformat{\subsubsection}[runin]
{\em}{\thesubsubsection{.}}{0.15cm}{}[.]

\usepackage[up,bf]{caption}

\newtheorem{theorem}{Theorem}[section]
\newtheorem{proposition}[theorem]{Proposition}
\newtheorem{claim}[theorem]{Claim}
\newtheorem{lemma}[theorem]{Lemma}
\newtheorem{corollary}[theorem]{Corollary}
\newtheorem{remark}[theorem]{Remark}

\theoremstyle{definition}

\numberwithin{equation}{section}
\numberwithin{figure}{section}

\usepackage{color}

\begin{document}

\fancyhead[CO]{The Minkowski problem and constant curvature surfaces} 
\fancyhead[CE]{A. Alarc\'{o}n$\,$  and$\,$ R. Souam} 
\fancyhead[RO,LE]{\thepage} 

\thispagestyle{empty}

\vspace*{1cm}
\begin{center}
{\bf\LARGE The Minkowski problem, new constant curvature surfaces in $\r^3,$ and some applications}

\vspace*{0.5cm}

{\large\bf Antonio Alarc\'{o}n$\;$ and$\;$ Rabah Souam}
\end{center}

\footnote[0]{\vspace*{-0.4cm}

\noindent A. Alarc\'{o}n

\noindent Departamento de Geometr\'{\i}a y Topolog\'{\i}a, Universidad de Granada, E-18071 Granada, Spain.

\noindent e-mail: {\tt alarcon@ugr.es}

\vspace*{0.1cm}

\noindent R. Souam

\noindent Institut de Math\'{e}matiques de Jussieu-Paris Rive Gauche,   UMR 7586, B\^{a}timent Sophie Germain,  Case 7012, 75205  Paris Cedex 13, France.

\noindent e-mail: {\tt souam@math.jussieu.fr}

\vspace*{0.1cm}

\noindent A. Alarc\'{o}n is supported by Vicerrectorado de Pol\'{i}tica Cient\'{i}fica e Investigaci\'{o}n de la Universidad de Granada.

\noindent A. Alarc\'{o}n's research is partially supported by MCYT-FEDER research projects MTM2007-61775 and MTM2011-22547, and Junta de Andaluc\'{i}a Grant P09-FQM-5088.}

\vspace*{1cm}

\begin{quote}
{\small
\noindent {\bf Abstract}\hspace*{0.1cm} Let $\mgot\in\n,$ $\mgot\geq 2,$ and let $\{p_j\}_{j=1}^\mgot$ be a finite subset of $\s^2$ such that $\vec{0}\in\r^3$ lies in its positive convex hull. In this paper we make use of the classical Minkowski problem, to show the complete family of smooth convex bodies $\Kscr$ in $\r^3$ whose boundary surface consists of an open surface $S$ with constant Gauss curvature (respectively, constant mean curvature) and $\mgot$ planar compact discs $\overline{D}_1,\ldots,\overline{D}_\mgot,$ such that the Gauss map of $S$ is a homeomorphism onto $\s^2-\{p_j\}_{j=1}^\mgot$ and  $D_j\bot p_j,$ for all $j.$

We derive applications to the  generalized Minkowski problem, existence of harmonic diffeomorphisms between domains of $\s^2,$ existence of capillary surfaces in $\r^3,$ and a Hessian equation of Monge-Amp\`ere type.

\vspace*{0.2cm}

\noindent{\bf Keywords}\hspace*{0.1cm} Constant Gauss curvature surfaces, constant mean curvature surfaces, harmonic diffeomorphisms between surfaces, Minkowski's problem, capillary surfaces, Monge-Amp\`ere equations.

\vspace*{0.2cm}

\noindent{\bf Mathematics Subject Classification (2010)}\hspace*{0.1cm} 53C42, 53C43, 53C21, 53A10.
}
\end{quote}


\section{Introduction}\label{sec:intro}

Let $\s^2$ and $\overline{\c}$ denote the $2$-dimensional  Euclidean unit sphere and the Riemann sphere, respectively. A domain in $\overline{\c}$ is said to be a circular domain if every connected component of its boundary is a circle.

In \cite{AS}, circular domains $U$ and harmonic diffeomorphisms $U\to \s^2-\{p_1,\ldots,p_\mgot\}$ were shown, where $\{p_1,\ldots,p_\mgot\}\subset\s^2$ is an arbitrary subset with cardinal number $\mgot\in\n,$ $\mgot\geq 2.$ Such diffeomorphisms were constructed as vertical projection of maximal graphs over $\s^2-\{p_1,\ldots,p_\mgot\}$ in the Lorentzian product manifold $\s^2\times\r_1.$ On the other hand, the Gauss map of constant mean curvature surfaces in $\r^3$ (from now on, {\em $H$-surfaces}) is harmonic for the conformal structure induced by isothermal charts \cite{Ru}, whereas the Gauss map of positive constant Gauss curvature surfaces in $\r^3$ (from now on, {\em $K$-surfaces}) is harmonic for the conformal structure of the second fundamental form (from now on, the {\em extrinsic conformal structure}) \cite{GM}. Therefore, given $\{p_1,\ldots,p_\mgot\}\subset\s^2,$ $\mgot\geq 2,$ the following questions naturally arise:
{\em
\begin{enumerate}[\sf ({Q}1)]
\item[\sf ({Q}$_H$)] Do there exist $H$-surfaces whose conformal structures are circular domains $U$ and their  Gauss maps  harmonic diffeomorphisms $U\to \s^2-\{p_1,\ldots,p_\mgot\}$?
\item[\sf ({Q}$_K$)] Do there exist $K$-surfaces whose extrinsic conformal structures are circular domains $U$ and their Gauss maps  harmonic diffeomorphisms $U\to \s^2-\{p_1,\ldots,p_\mgot\}$?
\end{enumerate}
}

Since for the unit complex disc $\d$ there is no harmonic diffeomorphism $\d\to\s^2-\{p\}$ (see \cite{AS}), then the answer to both {\sf (Q$_H$)} and {\sf (Q$_K$)} is negative when $\mgot=1.$ For $\mgot=2$ and $p_2=-p_1,$ it is well known that the answer to both questions is positive and the solution surfaces are of revolution; more precisely, the solution $H$-surfaces  are pieces of nodoids; see \cite{Sp} and Figure \ref{fig:rotational}. As far as the authors know, there is no other available existence result regarding these two questions.

On the other hand, since $K$-surfaces are parallel surfaces to $H$-surfaces, then a positive answer to {\sf (Q$_K$)} would imply the same to {\sf (Q$_H$)}.  In general $H$-surfaces are not  parallel to $K$-surfaces. However  an $H$-surface with non vanishing Gauss curvature is parallel to a $K$-surface (see the proof of Corollary \ref{co:H}) and hence the reciprocal assertion  holds too; observe that the Gauss curvature of a surface whose Gauss map is a local diffeomorphism has no zeros. Therefore, questions {\sf (Q$_H$)} and {\sf (Q$_K$)} are actually equivalent.

The aim of this paper is to settle question {\sf (Q$_K$)} above. Obviously we can assume without loss of generality that $K=1.$ We show the following  classification result:

\begin{theorem}\label{th:intro}
Let $\{p_1,\ldots,p_\mgot\}$ be a subset of $\s^2$ with cardinal number $\mgot\in\n.$

The following statements are equivalent:
\begin{enumerate}[\sf (i)]
\item There exists a $K$-surface $S$ with $K=1$ such that the extrinsic conformal structure of $S$ is a circular domain $U\subset\overline{\c},$ and the Gauss map of $S$ is a harmonic diffeomorphism $U\to\s^2-\{p_1,\ldots,p_\mgot\}.$ 
\item There exist positive real constants $a_1,\ldots,a_\mgot$ such that $\sum_{j=1}^\mgot a_j p_j=\vec{0}\in\r^3.$
\end{enumerate}

Furthermore, if $S$ is as above and one denotes by $\gamma_j$ the connected component of $\overline{S}-S$ corresponding to $p_j$ via its Gauss map, then
\begin{enumerate}[\sf (I)]
\item $\gamma_j$ is a Jordan curve contained in an affine plane $\Pi_j\subset\r^3$ orthogonal to $p_j,$ and
\item $\Sscr=S\cup (\cup_{j=1}^\mgot \overline{D}_j)$ is the boundary surface of a smooth convex body$^1${\footnote[0]{$^1$We use here a standard terminology for convex bodies: a convex body $\Kscr$ in $\r^3$ is said smooth if it has a unique supporting plane at each boundary point. This is the same as saying that $\partial \Kscr$ is a $\Ccal^1$ surface.}} in $\r^3,$ where $D_j$ is the bounded connected component of $\Pi_j-\gamma_j$ for all $j\in\{1,\ldots,\mgot\}.$
\end{enumerate}
In addition, given $\{a_1,\ldots,a_\mgot\}$ satisfying {\sf (ii)}, there exists a unique, up to translations, surface $S$ satisfying {\sf (i)} such that the area of $D_j$ equals $a_j$ for all $j\in\{1,\ldots,\mgot\}.$
\end{theorem}

\begin{figure}[ht]
    \begin{center}
    \scalebox{0.25}{\includegraphics{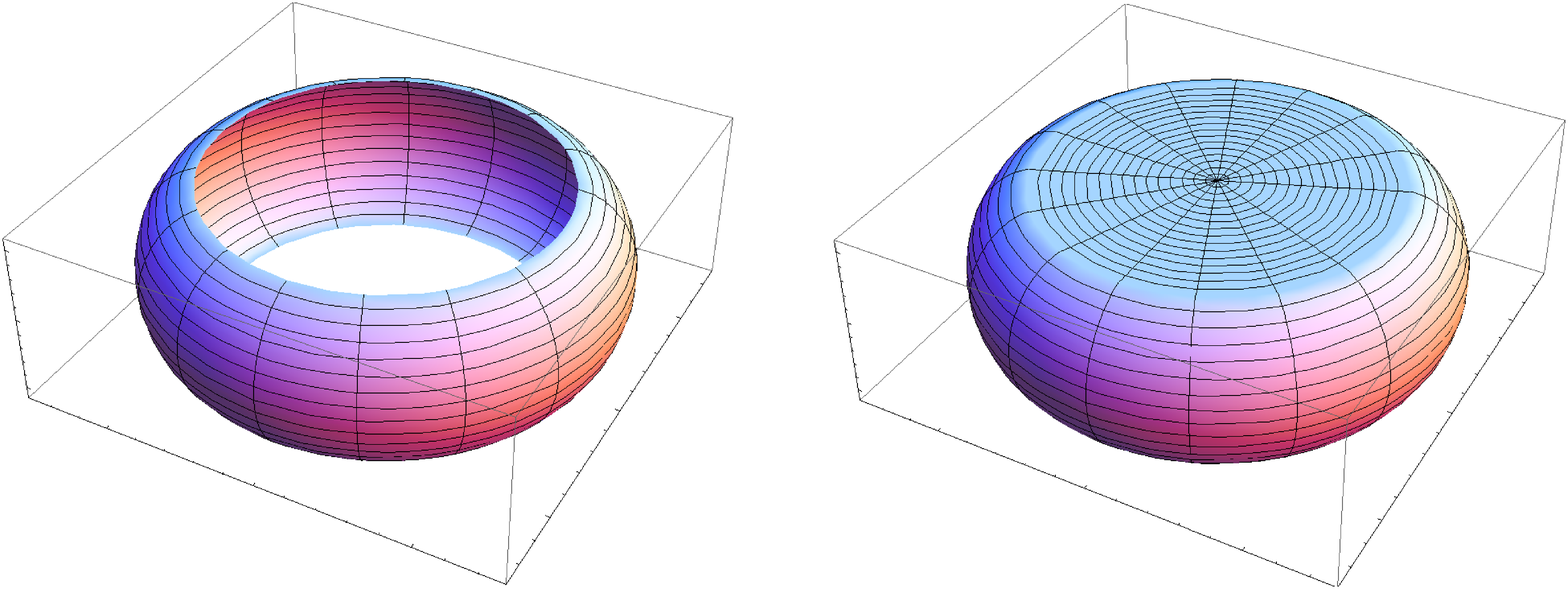}}
        \end{center}
        \vspace{-0.25cm}
\caption{Surfaces $S$ and $\Sscr$ in Theorem \ref{th:intro} for $\{p_1,p_2=-p_1\}$ and $a_1=a_2=4;$ see \cite{Sp}.}\label{fig:rotational}
\end{figure}

{ Our main tool to prove the existence part of Theorem \ref{th:intro} (i.e., {\sf (ii)}$\Rightarrow${\sf (i)}) is the classical Minkowski problem of prescribing positive Gauss curvature on the sphere; see Sect. \ref{sec:Minkowski} for a good setting. More precisely,}
our approach to show the $K$-surface $S$ in {\sf (i)} roughly goes as follows. For any $n\in\n$ we construct a smooth convex body $\Kscr_n$ in $\r^3$ whose boundary surface consists of an open surface $S_n$ of constant Gauss curvature $K=1,$ and $\mgot$ compact discs $S_{n,1},\ldots,S_{n,\mgot},$ such that the Gauss map of $S_n$ is a homeomorphism onto  {the complement of an} $1/n$-neighborhood of $\{p_1,\ldots,p_\mgot\}$ in $\s^2$ and the Gauss curvature of $S_{n,j}$ is smaller or equal than $1,$ for all $j.$  Furthermore, any disc $S_{n,j}$ contains another disc $D_{n,j}$ with area equal to $a_j$ and constant Gauss curvature $K\approx 1/n^2.$ The convex body $\Kscr_n$ is obtained as solution to the Minkowski problem, involving the equilibrium condition {\sf (ii)}. Then, we prove that the sequence $\{\Kscr_n\}_{n\in\n}$ has a smooth limit convex body  $\Kscr$ and obtain the surface $\Sscr$ in the theorem as the boundary surface of $\Kscr.$

Notice that the convex surface $\Sscr$ in the statement of Theorem \ref{th:intro} agrees with the solution to the generalized Minkowski problem for  the Borel measure 
\begin{equation}\label{eq:int0gen}
\mu(\Sscr)=\mu_{\s^2}+\sum_{j=1}^\mgot a_j \delta_{p_j}\quad \text{on $\s^2,$}
\end{equation}
 where $\mu_{\s^2}$ denotes the canonical Lebesgue measure on $\s^2$ and $\delta_{p_j}$ the Dirac measure at $p_j;$
see Sect. \ref{sec:Minkowski} for details. However, after solving the problem for $\mu(\Sscr),$ one knows nothing about the regularity of the solution. In fact, determining the regularity of the solution to the Minkowski problem depending on the one of the curvature function has been the key question of the topic; see \cite{L,P,N,CY}. A significant consequence of Theorem \ref{th:intro} is that the solution to the generalized Minkowski problem for \eqref{eq:int0gen} is a $\Ccal^1$  and piecewise analytic surface.

Moreover, Theorem \ref{th:intro} has interesting applications concerning capillary surfaces in $\r^3,$ harmonic diffeomorphisms between domains of $\s^2,$ and a Hessian equation of Monge-Amp\`ere type. We explain this in Sect. 5.

\subsection{Acknowledgments} The authors are grateful to Jos\'{e} A. G\'{a}lvez for suggesting them problem {\sf (Q$_K$)} and for helpful discussions about the paper.


\section{Preliminaries}\label{sec:preli}

As usual, we denote by $\|\cdot\|$ and $\langle\cdot,\cdot\rangle$ the Euclidean norm and inner product in $\r^n,$ $n\in\n.$ Given a subset $C\subset\r^3,$ we denote by $\overline{C}$ the closure of $C$ in $\r^3.$ 

Throughout the paper, unless otherwise specified, by a surface we mean an orientable surface with empty boundary; in particular, a surface is either open or compact. For an open surface $S$ in $\r^3,$ we denote by $\partial S$ the set determined by the frontier points of $\overline{S};$ i.e., $\partial S=\overline{S}-S.$

A {\em convex body} in $\r^3$ is a compact convex subset of $\r^3$ having interior points. A {\em smooth} convex body is a convex body which has a unique supporting plane at every boundary point; this is equivalent to the boundary surface being $\Ccal^1.$ A {\em strictly convex body} is a convex body whose boundary
does not contain any nontrivial line segment. Finally, a compact surface in $\r^3$ is said to be {\em (strictly) convex} if it is the boundary surface of a smooth (strictly) convex body in $\r^3.$

\subsection{Constant curvature surfaces}\label{sub:cmc}

Let $M$ be a smooth surface and let $X:M\to\r^3$ be an immersion with positive constant Gauss curvature $K.$ Without loss of generality we assume that $K=1,$ and from now on such an immersion $X$ is said to be a {\em $K$-immersion}, and its image surface $X(M)$ is said to be a {\em $K$-surface}. Up to changing orientation if necessary, the second fundamental form $II_X$ of $X$ is a positive definite metric. Therefore, $II_X$ induces on $M$ a conformal structure $M^X$ which is said to be the {\em extrinsic conformal structure} of $X$ (and of $X(M)$ as well). Then $X$ may be understood as an immersion
$X:M^X\to\r^3$ and the equation $K = 1$
implies that the unit normal vector field $N_X:M^X\to\s^2$ of $X$ is a harmonic local diffeomorphism; see \cite{GM}.

Any $K$-surface $S\subset\r^3$ is {\em locally strictly convex}; i.e., for any $p\in S$ there exists an open neighborhood $U_p$ of $p$ in $S$ such that $U_p\cap T_pS=\{p\},$ where $T_pS$ denotes the affine tangent plane to $S$ at $p.$

On the other hand, given an immersion $Y:M\to\r^3$ with constant mean curvature $H,$ its first fundamental form $I_Y$ induces on $M$ a conformal structure $M_Y.$ The Riemann surface $M_Y$ is said to be the {\em intrinsic conformal structure} of $Y$ (and of $Y(M)$). In this case, if one considers $Y:M_Y\to\r^3,$ then the Gauss map $N_Y:M_Y\to\s^2$ of $Y$ is harmonic as well; see \cite{Ru}. From now on, such an immersion $Y$ with $H=1/2$ and its image surface $Y(M)$ are said to be an {\em $H$-immersion} and an {\em $H$-surface}, respectively.

The following well known connection between $K$-surfaces and $H$-surfaces is very useful in this paper. Let $S$ be a $K$-surface and let $N_S:S\to\s^2\subset\r^3$ be its {\em outer Gauss map}; that is to say, the one that at any $p\in S$ points to the connected component of $\r^3-T_pS$ disjoint from an open neighborhood of $p$ in $S.$ Then
\[
S+N_S:=\{p+N_S(p)\;|\; p\in S\}
\]
is a locally strictly convex $H$-surface with outer Gauss map $N_{S+N_S}(q)=N_S(p)$ for any $q=p+N_S(p)\in S+N_S.$ The surface $S+N_S$ is said to be the {\em outer parallel surface} to $S$ at distance $1.$ Furthermore, the extrinsic conformal structure of $S$ and the intrinsic conformal structure of $S+N_S$ are biholomorphic.

\subsection{The Minkowski problem}\label{sec:Minkowski}

Let $X:\s^2\to\r^3$ be an immersion such that its image surface $X(\s^2)$ is a closed strictly convex surface in $\r^3.$ Then the Gauss map $N_X:\s^2\to \s^2$ of $X$ is a homeomorphism. Define $\kappa:\s^2\to\r,$ $\kappa=K\circ N_X^{-1},$ where $K:\s^2\to\r$ denotes the Gauss curvature function of $X.$ In this setting, Minkowski observed that $\kappa$ must satisfy
\begin{equation}\label{eq:minkowski}
\int_{\s^2} \frac{p}{\kappa(p)}\, dp=0. 
\end{equation}

The converse of this problem is known as the ($2$-dimensional) Minkowski problem and there is a large literature dealing with it. It turns out that actually the condition \eqref{eq:minkowski} is necessary and sufficient. For our purposes, we will use the following classical result \cite{N, P}; see also \cite{CY}.

\begin{theorem}\label{th:minkowski}
Let $\kappa:\s^2\to\r$ be a  smooth positive function satisfying \eqref{eq:minkowski}.

Then there exists a unique up to translations smooth embedding $X:\s^2\to\r^3$ such that $X(\s^2)$ is a closed strictly convex surface and the curvature function $K:\s^2\to\r$ of $X$ is given by $K=\kappa\circ N_X,$ where $N_X:\s^2\to\s^2$ denotes the Gauss map of $X.$
\end{theorem}

Obviously, in the setting of Theorem \ref{th:minkowski}, the map $X\circ N_X^{-1}:\s^2\to \r^3$ is an immersion with Gauss curvature function $\kappa$ and Gauss map the identity map of $\s^2.$

Let us introduced now the {\em generalized ($2$-dimensional) Minkowski problem;} see \cite{CY,Sc} for a good setting. Let $S$ be a compact convex surface in $\r^3,$ not necessarily smooth; i.e., $S$ is the boundary of a general convex body in $\r^3.$ The {\em generalized Gauss map} $G:S\to\s^2$ of $S$ is a set-valued map, mapping every point $p\in S$ to the set of all outer normals of the supporting planes of $S$ passing through $p.$ With this map in hand, one can define a measure $\mu(S)$ on $\s^2$ called the {\em area function} of $S$ by setting
\[
\mu (M,E)={\rm Area}(\{p\in S\,|\, G(p)\cap E\neq \emptyset\})\quad \text{for any Borel subset $E\subset\s^2,$}
\]
where ${\rm Area}(\cdot)$ denotes the area functional. For instance, if $S$ is $\Ccal^2$ and strictly convex, then $\mu(S)$ is nothing but $\mu_{\s^2}/\kappa,$  where $\mu_{\s^2}$ denotes the canonical Lebesgue measure on $\s^2$ and $\kappa$ the Gauss curvature of $S$ transplanted to $\s^2$ via the Gauss map. On the other hand, if $S$ is a polyhedron, then $\mu(S)=\sum_{j=1}^n c_j\, \delta_{\nu_j},$ where $\delta_{\nu_j}$ is the Dirac measure at $\nu_j$ (i.e.; the unit point mass) and $c_j$ is the Euclidean area of the face of $S$ with outer normal $\nu_j.$

The uniqueness part of Theorem \ref{th:intro} will be obtained from the following result due to Minkowski, Alexandrov, Fenchel, and Jessen; see \cite{Br,CY}.
\begin{theorem}\label{th:MinGen}
Let $\mu$ be a non-negative Borel measure on $\s^2$ such that $\int_{\s^2} {\rm i}_{\s^2} \,\mu=\vec{0}\in\r^3$ and $\mu(H)>0$ for any open hemisphere $H\subset\s^2,$ where ${\rm i}_{\s^2}:\s^2\to\r^3$ denotes the inclusion map.

Then there exists a convex body $\Kscr$ in $\r^3$ such that $\mu=\mu(\partial \Kscr).$ Furthermore, $\Kscr$ is unique up to translations.
\end{theorem}

Observe that the measure $\mu(\Sscr)$ induced on $\s^2$ by the convex surface $\Sscr$ in the statement of Theorem \ref{th:intro} is given by 
\[
\mu(\Sscr)=\mu_{\s^2}+\sum_{j=1}^\mgot a_j \delta_{p_j},
\]
hence, taking into account the equilibrium condition {\sf (ii)}, Theorem \ref{th:MinGen} applies. However, this does not provide any information on the regularity of the solution.

\subsection{The support function}\label{sec:support}

Let $\Sigma$ be an open domain of $\s^2$ and let $X:\Sigma\to\r^3$ be a smooth  immersion whose Gauss map $N_X:\Sigma\to\s^2$ is a diffeomorphism into its image $N_X(\Sigma)\subset\s^2.$ Then the support function $h:N_X(\Sigma)\to\r$ of $X$ is defined by 
\begin{equation}\label{eq:support0}
h(p):=\max_{x\in\Sigma}\, \langle p, X(x) \rangle=\langle p, X(N_X^{-1}(p))\rangle.
\end{equation}
Denote by $K$ the curvature function of $X.$ Then $h$ satisfies
\[ 
\big(\det\big( \nabla^2 h + h\, {\rm I} \big)\big)\circ N_X=\frac{1}{K}\quad\text{on $\Sigma,$}
\] 
where $\nabla^2h(p) $ and ${\rm I}$ denote the Hessian matrix of $h$ at $p$ on $\s^2$  and the identity matrix of $T_p\s^2,$ for $p\in \s^2,$ respectively. Furthermore, $X$ can be recovered from $h$ in the form
\begin{equation}\label{eq:support-recover}
X\circ N_X^{-1} (p) =\nabla h (p) + h(p) \, p\quad\text{on $N_X(\Sigma),$}
\end{equation}
where $\nabla h(p)$ denotes the gradient of $h$ at $p$ computed with respect to the spherical metric and viewed as a vector in $\r^3.$ As in Sect. \ref{sec:Minkowski}, $X\circ N_X^{-1}:N_X(\Sigma)\to \r^3$ is an immersion with Gauss curvature function $K\circ N_X^{-1}$ and Gauss map the inclusion map of $N_X(\Sigma)$ into $\s^2.$


\section{Convexity and the equilibrium condition}\label{sec:EC}

Let $\{p_1,\ldots,p_\mgot\}$ be a subset of $\s^2$ with cardinal number $\mgot\in\n.$ In this section we assume the existence of a $K$-surface $S$ as those described by Theorem \ref{th:intro}-{\sf(i)} for $\{p_1,\ldots,p_\mgot\},$ and prove that $\{p_1,\ldots,p_\mgot\}$ and $S$ satisfy the equilibrium condition {\sf(ii)} and properties {\sf(I)} and {\sf(II)}, respectively.

Assume there exists a $K$-surface $S$ in $\r^3$ such that the extrinsic conformal structure of $S$ is conformally equivalent to a planar circular domain, and the outer Gauss map of $S$ is a diffeomorphism $N_S:S\to\s^2-\{p_1,\ldots,p_\mgot\}.$ 

First of all let us show that
\begin{claim}\label{cla:bounded}
$S$ is bounded.
\end{claim}
\begin{proof}
Fix $j\in\{1,\ldots,\mgot\}$ and up to a rigid motion assume that $p_j=(0,0,1)\in\r^3.$ Take $\epsilon >0,$ denote by 
\[
H_\epsilon=\{p\in S\,|\, \langle N_S(p), (0,0,1)\rangle\geq  1-\epsilon\}=\{p\in S\,|\, 1>\langle N_S(p), (0,0,1)\rangle\geq 1-\epsilon\},
\]
and notice that it suffices to show that $H_\epsilon$ is bounded; recall that $N_S:S\to\s^2-\{p_1,\ldots,p_\mgot\}$ is a diffeomorphism. Assume that $\epsilon$ is small enough so that $H_\epsilon$ is a topological annulus with boundary and a local graph in the $x_3$-direction at any point, where $x_3$ denotes the third coordinate function in $\r^3.$ Note that, since $N_S$ is the outer Gauss map of $S,$ the (local) defining function of the local graph is concave. Write $(X_1,X_2,X_3):H_\epsilon\to\r^3$ the inclusion map, and $(N_S)|_{H_\epsilon}=(N_1,N_2,N_3):H_\epsilon\to\s^2\cap\{1>x_3\geq 1-\epsilon\}\subset\r^3.$ In this setting, the Legendre transform of $H_\epsilon,$
\begin{equation}\label{eq:legendre}
\Lcal=\left(\frac{N_1}{N_3} \,,\, \frac{N_2}{N_3} \,,\, \frac{N_1}{N_3} X_1+ \frac{N_2}{N_3}X_2 + X_3\right):H_\epsilon\to\r^3,
\end{equation}
defines a {\em strongly positively curved} surface (i.e., with Gauss curvature bounded from below by a positive constant) $\Lcal(H_\epsilon)$ with boundary in $\r^3$ which is, moreover,  the local graph in the $x_3$-axis direction of a {\it convex}  function around any point; see for instance \cite{GMi,LSZ} and take into account that $H_\epsilon$ is a surface with boundary and constant Gauss curvature $K=1$ (hence strongly positively curved). 

Denote by $\Omega=\{x\in\r^2\,|\, \|x\|^2\leq \frac{2\epsilon-\epsilon^2}{(1-\epsilon)^2}\}$ and notice that the map 
\[
\s^2\cap\{1>x_3\geq 1-\epsilon\}\to \Omega-\{(0,0)\}, \quad (x_1,x_2,x_3)\mapsto(x_1/x_3,x_2/x_3),
\]
is a diffeomorphism. Since also $(N_1,N_2,N_3):H_\epsilon\to\s^2\cap\{1>x_3\geq 1-\epsilon\}$ is a diffeomorphism, then $\Lcal(H_\epsilon)$ is the graph of a convex function $\varphi_\Lcal:\Omega-\{(0,0)\}\to\r;$ see \eqref{eq:legendre}. Furthermore, since $\Lcal(H_\epsilon)$ is strongly positively curved, then $\varphi_\Lcal$ extends continuously to $\Omega,$ with the same name, and its graph $\overline{\Lcal(H_\epsilon)}$ is a strictly convex $\Ccal^0$ surface with boundary; see \cite{NR}.

Let us show first that $(X_1,X_2):H_\epsilon\to \r^2$ is bounded. Indeed, otherwise the limit set of the Gauss map of $\Lcal;$ which is given by
\begin{equation}\label{eq:legendre2}
N_\Lcal:H_\epsilon\to\s^2,\quad N_\Lcal=\frac{1}{\sqrt{X_1^2+X_2^2+1}}(X_1,X_2,-1)
\end{equation}
(see \cite{GMi,LSZ}), would contain an horizontal limit vector when { $(x_1,x_2)|_{\Lcal(H_\epsilon)}$} goes to $(0,0),$ hence $\overline{\Lcal(H_\epsilon)}$ would admit a vertical supporting plane at $(0,0,\varphi_\Lcal(0,0));$ here $x_1$ and $x_2$ denote the first and second coordinate functions in $\r^3.$ However, this contradicts that $\overline{\Lcal(H_\epsilon)}$ is a convex graph over $\Omega.$

Finally, since $(X_1,X_2):H_\epsilon\to \r^2$ is bounded and $\varphi_\Lcal$ extends continuously to $(0,0),$ then \eqref{eq:legendre} implies that  \begin{equation}\label{eq:x3}
\text{$X_3:H_\epsilon\to\r$ has a limit $(=\varphi_{\Lcal}(0,0))$ as 
$N_S\to (0,0,1).$}
\end{equation}
This concludes the proof.
\end{proof}

Now we can prove the following

\begin{claim}\label{cla:C1}
$\overline{S}$ is a $\Ccal^1$ surface with boundary. Furthermore, the connected component $\gamma_j$ of $\partial S$ corresponding to $p_j$ is a $\Ccal^1$ convex curve lying in a plane $\Pi_j$ in $\r^3$ orthogonal to $p_j,$ for all $j\in\{1,\ldots,\mgot\}.$ 
\end{claim}
\begin{proof}
Fix $j\in\{1,\ldots,\mgot\}$ and up to a rigid motion assume that $p_j=(0,0,1)\in\r^3.$ Let $\epsilon>0$ be small enough and let ${V}=H_{\epsilon}$  be the open set in $S$ introduced in the proof of Claim \ref{cla:bounded}.  We also let $\Omega$ and $\varphi_\Lcal:\Omega-\{(0,0)\}\to\r$ be the closed disc in $\r^2$ and the strictly convex function stated in Claim \ref{cla:bounded}.
We know from \eqref{eq:x3} that 
 $\gamma_j$  lies on the affine plane $\Pi_j=\{x\in\r^3\,|\, x_3=\varphi_{\Lcal}(0,0)\}$. 

Let us show that $\gamma_j$ is a $\Ccal^1$ Jordan curve bounding a convex disc $D_j$ in $\Pi_j.$ Indeed, 
for any $(x_1,x_2)\in \Omega-\{(0,0)\},$ since $\Lcal(H_{\epsilon})$ is  the graph of the strictly convex function $\varphi_{\Lcal},$ the affine tangent plane 
to $\Lcal(H_{\epsilon})$ at $(x_1,x_2,\varphi_{\Lcal}(x_1,x_2))$ lies below $\Lcal(H_{\epsilon}),$ so:
\begin{equation*}
\varphi_{\Lcal}(x_1,x_2) - x_1\frac{\partial \varphi_{\Lcal}}{\partial x_1 }(x_1,x_2) - x_2\frac{\partial \varphi_{\Lcal}}{\partial x_2} (x_1,x_2) <  \varphi_{\Lcal}(0,0).
\end{equation*}

Taking into account  (\ref {eq:legendre}) and (\ref{eq:legendre2}), this means that
\[ X_3 < \varphi_{\Lcal}(0,0) \quad \text{on $V.$}
\]

Otherwise said, $V$ lies below the plane $\Pi_j.$ Together with the fact that  $(N_S)|_V$ is a diffeomorphism onto a small open punctured neighborhood of $p_j$ in $\s^2,$ this shows that $\gamma^\epsilon=V\cap \Pi^\epsilon$ is a regular curve for any small enough $\epsilon>0,$ where $\Pi^\epsilon$ denotes the inner parallel plane to $\Pi_j$ at distance $\epsilon.$
Denote by $\nu^\epsilon$ the orthogonal projection of $N_S$ to $\Pi^\epsilon$ along $\gamma^\epsilon$ and notice that $\nu^\epsilon$ is a nowhere vanishing orthogonal vector field to $\gamma^\epsilon.$ It is easy to check that the curvature of the planar curve $\gamma^\epsilon$ is equal to 
$\kappa_n (\gamma^\epsilon) / \|\nu^\epsilon\|,$ where $\kappa_n (\gamma^\epsilon)$ is the normal curvature of the curve $\gamma^\epsilon$ in $S.$ This shows this curvature never vanishes and so $\gamma^\epsilon$ is locally convex.
Since $(N_S)|_V$ is injective into a small open punctured neighborhood of $p_j$ in $\s^2,$ it follows that the rotation index of $\gamma^\epsilon$ is $\pm 1.$   Therefore $\gamma^\epsilon$ is convex for all $\epsilon >0$ small enough. This trivially shows that $\gamma_j$ bounds a convex disc $D_j$ in $\Pi_j$ as claimed.

Set 
\begin{equation}\label{eq:Sscr}
\Sscr=S\cup(\cup_{j=1}^\mgot \overline{D}_j)
\end{equation}
and observe that $\Sscr$ is a closed locally convex $\Ccal^0$ surface; recall that the outer Gauss map $N_S$ of $S$ extends continuously to $\Sscr$ setting $(N_\Sscr)|_{\overline{D}_j}=p_j.$ Let us show that $\Sscr$ is (globally) convex; that is to say, at each point $\Sscr$ is contained in one side of its affine tangent plane. Indeed, let $p\in S.$ The affine plane $\lambda N_S(p)+T_pS$ is disjoint from $\Sscr$ for any large enough $\lambda>0;$ recall that $\Sscr$ is compact. Call $\lambda_p$ the biggest positive real $\lambda$ such that $(\lambda N_S( p)+T_pS)\cap \Sscr\neq\emptyset.$ For any point $q$ in $(\lambda_pN_S(p)+T_pS)\cap \Sscr,$ one has $N_S(q)=N_S(p),$ and so $q=p$ since $N_S$ is injective. In particular, $\Sscr$ lies in one side of $T_pS.$ A similar argument works at points in $\cup_{j=1}^\mgot \overline{D}_j.$ This shows that $\Sscr$ is convex as claimed. Since, in addition, $\Sscr$ has a unique supporting plane at every point, then $\Sscr$ bounds a smooth convex body and $\Sscr$ is $\Ccal^1.$
This proves the claim.
\end{proof}

In this way, we have shown that the convex surface $\Sscr$ given by \eqref{eq:Sscr} and so the $K$-surface $S$ are embedded.  Furthermore, since $\Sscr$ is convex and $\Ccal^1,$ then items {\sf (I)} and {\sf (II)} in Theorem \ref{th:intro} follow.

To finish this section, let us check that the set $\{p_1,\ldots,p_\mgot\}$ satisfies the equilibrium condition given by Theorem \ref{th:intro}-{\sf (ii)}. This easily follows from the following more general result.

\begin{proposition}
Let $M$ be an embedded $H$-surface in $\r^3$ and assume that $\Mscr=M\cup(\cup_{j=1}^\mgot \overline{C}_j)$ is a compact embedded $\Ccal^1$ surface, where $C_j$ is a disc in a plane $\Pi_j$ orthogonal to $q_j\in\s^2$ with $C_j\cap M=\emptyset$ for all $j\in\{1,\ldots,\mgot\},$ $\mgot\in\n.$ 

Then $\sum_{j=1}^\mgot {\rm Area}(C_j)\, q_j=\vec{0}.$
\end{proposition}
\begin{proof}
Denote by $N_\Mscr:\Mscr\to\s^2$ the outer Gauss map of $\Mscr$ as compact embedded surface, and notice that 
\begin{equation}\label{eq:pro}
(N_\Mscr)|_{C_j}=q_j\quad\text{for all $j.$}
\end{equation}

 Since $M$ is an $H$-surface, then the inclusion map ${\rm i}_M:M\to\r^3$ satisfies $\Delta {\rm i}_M=2H(N_\Mscr)|_M=(N_\Mscr)|_M,$ where $\Delta\cdot$ denotes the Laplace operator computed with respect to isothermal coordinates on $M;$ recall that $H=1/2.$ Then, the Divergence Theorem gives
\[
\int_M \langle N_\Mscr(p),x\rangle\,dp = \int_{\partial M} \langle \nabla \langle p,x\rangle\,,\, \eta(p)\rangle\, dp\quad\text{for all $x\in\r^3,$}
\]
where $\eta$ denotes the outer conormal vector field to $\partial M$ on $M.$ Therefore,
\begin{equation}\label{eq:int0}
\int_M N_\Mscr(p)\,dp=\int_{\partial M}\eta(p)\, dp=\vec{0}\in\r^3,
\end{equation}
where for the second equality we have used that $\partial M$ consists of plane curves.

On the other hand, since $\Mscr$ is a compact embedded $\Ccal^1$ surface, then
\begin{equation}\label{eq:int0b}
\vec{0}=\int_\Mscr N_\Mscr(p)\,dp = \int_M N_\Mscr(p)\,dp + \sum_{j=1}^\mgot \int_{C_j} N_\Mscr(p)\,dp.
\end{equation}
Combining \eqref{eq:int0b}, \eqref{eq:int0}, and \eqref{eq:pro}, one infers
\[
\sum_{j=1}^\mgot {\rm Area}(C_j)\, q_j=\vec{0},
\]
as claimed.
\end{proof}

To obtain the condition Theorem \ref{th:intro}-{\sf (ii)} it suffices to apply the above proposition to the outer parallel surface to $S$ at distance $1$ and set $a_j={\rm Area}(D_j)>0$ for all $j.$

\begin{remark}
The equilibrium formula {\sf (ii)} also follows directly from the facts that $\int_S N_S(p)\, dp=\int_{\s^2-\{p_1,\ldots,p_\mgot\}}p\, dp=\vec{0}$ and $\int_\Sscr N_\Sscr(p)\, dp=\vec{0}.$
\end{remark}


\section{Existence}\label{sec:existence}

Let $\{p_1,\ldots,p_\mgot\}$ be a subset of $\s^2$ with cardinal number $\mgot\in\n.$ In this section, assuming that the equilibrium condition Theorem \ref{th:intro}-{\sf(ii)} holds for $\{p_1,\ldots,p_\mgot\}$ and positive real constants $\{a_1,\ldots,a_\mgot\},$ we show the existence and uniqueness (up to translation) of a $K$-surface $S$ in $\r^3$ satisfying Theorem \ref{th:intro}-{\sf(i)}, {\sf (I)}, {\sf (II)}, and ${\rm Area}(D_j)=a_j,$ where $D_j$ are the discs given by Theorem \ref{th:intro}-{\sf(II)}, for all $j\in\{1,\ldots,\mgot\}.$ Taking into account the already done in Sect. \ref{sec:EC}, this will complete the proof of Theorem \ref{th:intro}.

Let $a_1,\ldots,a_\mgot\in\r$ such that $a_j>0$ for all $j,$ and assume that
\begin{equation}\label{eq:EC}
\sum_{j=1}^\mgot a_j\, p_j=\vec{0}\in\r^3.
\end{equation}

We need some preliminaries before going on into the construction process. Given $q\in\s^2$ and $r\in(0,1)$ we set
\[
B(q,r)=\{p\in\s^2\;|\; \sin\sphericalangle(p,q)<r,\; \cos\sphericalangle(p,q)>0\}
\]
and
\[
A(q,r/2)=B(q,r)-\overline{B(q,r/2)},
\]
where $\sphericalangle(p,q)\in[0,\pi]$ denotes the spherical angle between $p$ and $q.$ Straightforward computations give the following
\begin{claim}\label{cla:spherical}
Let $q\in\s^2$ and let $r\in (0,1/2).$ Then
\[
\int_{B(q,r)} p \, dp =\pi r^2 q.
\]
Moreover, for any $r\in(0,1/4),$ $\lambda >1,$ and $\mu\in (3\pi r^2,3\pi r^2\lambda),$ there exists a smooth function on an interval $[0,2r+\epsilon],$ $\epsilon>0,$ such that
$f|_{[0,r]}=\lambda,$ $f|_{[2r,2r+\epsilon]}=1,$ $1\leq f|_{(r,2r)}\leq\lambda,$ and
\[
\int_{A(q,r)} f(\sin\sphericalangle(p,q)) p\, dp =\mu q.
\]
\end{claim}

Fix $n_0\in\n$ large enough so that the sets $\overline{B(p_j,2/n)},$ $j=1,\ldots,\mgot,$ are pairwise disjoint, and
\begin{equation}\label{eq:n}
\frac{1}{n^2}<\frac{3a_j}{4\pi}\quad \text{for all $j\in\{1,\ldots,\mgot\}$ and $n\in\n$ with $n\geq n_0.$}
\end{equation}

For any $n\geq n_0,$ denote by $\Sigma_n=\s^2-\cup_{j=1}^\mgot \overline{B(p_j,2/n)}$ and let $f_n:\s^2\to \r$ be a  smooth function satisfying
\begin{equation}\label{eq:f1}
(f_n)|_{\Sigma_n}=1,\quad (f_n)|_{B(p_j,1/n)}=\frac{a_j n^2}{\pi},\quad 1\leq (f_n)|_{A(p_j,1/n)}\leq \frac{a_j n^2}{\pi},
\end{equation}
and
\begin{equation}\label{eq:f2}
\int_{A(p_j,1/n)} f_n(p) p\, dp = \frac{4\pi}{n^2} p_j = \int_{B(p_j,2/n)} p\, dp,
\end{equation}
for all $j.$ The existence of such a function $f_n$ follows directly from Claim \ref{cla:spherical} and \eqref{eq:n}. By Claim \ref{cla:spherical} and equations \eqref{eq:EC}, \eqref{eq:f1}, and \eqref{eq:f2}, one has
\begin{eqnarray*}
\int_{\s^2} f_n(p)p\, dp & = & \int_{\Sigma_n} f_n(p)p\, dp + \sum_{j=1}^\mgot \int_{A(p_j,1/n)} f_n(p) p\, dp +\sum_{j=1}^\mgot \int_{B(p_j,1/n)} f_n(p) p\, dp\\
& = & \int_{\Sigma_n} p\, dp + \sum_{j=1}^\mgot \int_{B(p_j,2/n)} p\, dp + \sum_{j=1}^\mgot \int_{B(p_j,1/n)} \frac{a_j n^2}{\pi} p\, dp \\
& = & \int_{\s^2} p\, dp + \sum_{j=1}^\mgot a_j\, p_j=\vec{0}+\vec{0}=\vec{0}.\\
\end{eqnarray*}

This shows that the Minkowski problem can be solved for the function $\kappa_n=1/f_n:\s^2\to\r;$ see \eqref{eq:minkowski}. Then Theorem \ref{th:minkowski} provides a smooth  embedding $X_n:\s^2\to\r^3$ such that
\begin{enumerate}[\sf (a{$_n$})]
\item $\Sscr_n:=X_n(\s^2)$ is a closed smooth strictly convex surface,
\item the Gauss map of $X_n$ is the identity map of $\s^2,$ and
\item the curvature function of $X_n$ agrees $\kappa_n.$
\end{enumerate}

Notice that {\sf (c$_n$)} and \eqref{eq:f1} give that $S_n:=X_n(\Sigma_n)$ is a $K$-surface. On the other hand, from {\sf (b$_n$)} and {\sf (c$_n$)}, one infers that
\begin{enumerate}[\sf (a{$_n$})]
\item[\sf (d{$_n$})] the area element of $X_n$ is $1/\kappa_n=f_n$ times the one of $\s^2,$
\end{enumerate} 
hence
\begin{equation}\label{eq:area}
{\rm Area}(X_n(B(p_j,1/n)))= a_j
\end{equation}
and
\begin{equation}\label{eq:areaAnillo}
{\rm Area}(X_n(A(p_j,1/n)))<2\int_{A(p_j,1/n)} f_n(p)\langle p,p_j\rangle\, dp=\frac{8\pi}{n^2}\quad \forall j=1,\ldots,\mgot;
\end{equation}
here we have used \eqref{eq:f1}, \eqref{eq:f2}, and that $\langle p,p_j\rangle>1/2$ for all $p\in B(p_j,2/n).$

Let $\Kscr_n$ denote the convex body bordered by $\Sscr_n.$ We plan to find the surface $\Sscr$ which solves the theorem, as the boundary surface of the limit of a subsequence of the sequence of convex bodies $\{\Kscr_n\}_{n\geq n_0}.$ For this limit to exist, a uniform upper bound of the diameter of $\Kscr_n,$ $n\geq n_0,$ is needed. Furthermore, to guarantee that the limit is a convex body as well, one also needs a uniform lower bound of the inner diameters. Denote by $\ell_n$ the extrinsic diameter of $\Kscr_n$ and, without loss of generality, assume that $\vec{0}$ is the middle point between two points $x_n$ and $y_n$ in $\Sscr_n$ at distance $\ell_n.$ The proof of the following technical lemma is an adaptation of results in \cite{CY}.

\begin{lemma}\label{lem:bounded}
There exist $\xi>0$ and $x\in \r^3$ such that, up to passing to a subsequence, $\b(x,\xi)\subset \Kscr_n\subset\b(\vec 0, 1/\xi)$ for all $n\in\n,$ $n\geq n_0,$ where $\b(y,r)$ denotes the euclidean ball in $\r^3$ of radius $r>0$ centered at $y.$
\end{lemma}
\begin{proof}
First let us show the existence of $\tau>0$ such that
\begin{equation}\label{eq:arriba}
\Kscr_n\subset \b(\vec 0, \tau)\quad\forall n\geq n_0.
\end{equation}

Set $u_n=(x_n-y_n)/\|x_n-y_n\|.$ Let $h_n:\s^2\to\r$ be the support function of $\Sscr_n.$ Then
\[
h_n(p)=\sup_{q\in\Sscr_n} \langle p,q\rangle \geq \frac{\ell_n}2 \max \big( 0\,,\, \langle p, u_n\rangle\big)
\]
 (see \eqref{eq:support0}), hence
\begin{equation}\label{eq:bound1}
\int_{\s^2} \frac{h_n(p)}{\kappa_n(p)}\, dp \geq \frac{\ell_n}{2}\int_{\s^2} \frac{\max \big( 0\,,\, \langle p, u_n\rangle\big)}{\kappa_n(p)}\, dp \geq c_0 \ell_n,
\end{equation}
where $c_0=2\int_{\s^2} \max \big( 0\,,\, \langle p, u_n\rangle\big)\, dp$ is a positive constant independent on $u_n\in\s^2$ and
we have used that $\kappa_n\leq 1$ on $\s^2;$ see \eqref{eq:f1}. Denote by $\hat{h}_n=h_n\circ X_n^{-1}:\Sscr_n\to \r.$ Taking into account {\sf (d$_n$)}, one gets
\begin{equation}\label{eq:bound2}
\int_{\s^2} \frac{h_n(p)}{\kappa_n(p)}\, dp = \int_{\Sscr_n} \hat{h}_n (p)\, dp=3{\rm vol}(\Kscr_n),
\end{equation}
where ${\rm vol}(\cdot)$ denotes volume in $\r^3;$ the second equality is a well known property of the support function. 

On the other hand, by the isoperimetric inequality in $\r^3,$
\begin{equation}\label{eq:bound3}
{\rm vol}(\Kscr_n)\leq c_1 {\rm Area}(\Sscr_n)^{3/2} 
\end{equation}
for a positive constant $c_1.$
Combining \eqref{eq:bound1}, \eqref{eq:bound2}, and \eqref{eq:bound3}, one has
\[ 
3 c_1 {\rm Area}(\Sscr_n)^{3/2} \geq  c_0\ell_n.
\] 
Therefore, to check that $\ell_n$ is uniformly bounded from above, it suffices to prove that ${\rm Area}(\Sscr_n)$ is bounded from above by a constant independent of $n.$ It follows from {\sf (d$_n$)} and \eqref{eq:f1}  that
${\rm Area}(X_n(\Sigma_n)) ={\rm Area}(\Sigma_n).$ Then, taking \eqref{eq:area} and \eqref{eq:areaAnillo} into account, one concludes that 
\[ 
{\rm Area}(\Sscr_n)<4\pi+ \sum_{j=1}^\mgot 8\pi/n^2 +\sum_{j=1}^\mgot a_j.
\] 
This implies that $\ell_n$ is uniformly bounded, hence \eqref{eq:arriba} holds for any $\tau>\sup_{n\geq n_0}\ell_n/2.$

To finish the proof, it suffices to find a ball in $\r^3$ contained in $\Kscr_n$ for all $n\geq n_0.$ Indeed, for any affine plane $\Pi\subset\r^3$ denote by $\Sscr_n(\Pi)$ the vertical projection of $\Sscr_n$ on $\Pi.$ Since trivially ${\rm vol}(\Kscr_n)\leq{\rm Area}(\Sscr_n(\Pi))\ell_n,$ then \eqref{eq:bound1} and \eqref{eq:bound2} give
\begin{equation}\label{eq:bound5}
{\rm Area}(\Sscr_n(\Pi))\geq c_0/3\quad\text{for all $n$ and $\Pi.$}
\end{equation}

 Argue by contradiction and assume that for any $i\in\n$ there exists $n_i\in\n,$ $n_i\geq n_0,$ such that $\Kscr_{n_i}$ contains no ball of radius $1/i.$ From the bound \eqref{eq:arriba}, Blaschke selection theorem \cite{Sc} implies that, up to passing to a subsequence, $\{\Kscr_n\}_{n\geq n_0}$ converges in the Hausdorff distance to a convex set $\Kscr_\infty\subset \b(\tau)$ which contains no Euclidean ball by our hypothesis. Then, $\Kscr_\infty$ is contained in an affine plane $\Pi_\infty,$ hence for any plane $\Pi_\infty^*$ orthogonal to $\Pi_\infty$ the sequence $\{{\rm Area}(\Sscr_n(\Pi_\infty^*))\}_{n\geq n_0}$ converges to zero, which contradicts \eqref{eq:bound5}. This shows the existence of $\rho>0$ such that for each $n\ge n_0,$ there exists $x_n \in \r^3$  so that $\b(x_n, \rho)\subset \Kscr_n.$ Up to taking a subsequence, $x_n$ converges to $x\in \r^3$ and $B(x,\rho/2)\subset \Kscr_n.$

This proves the lemma.
\end{proof}

By Lemma \ref{lem:bounded} and Blaschke selection theorem \cite{Sc}, up to passing to a subsequence, $\{\Kscr_n\}_{n\geq n_0}$ converges in the Hausdorff distance to a convex body $\Kscr$ in $\r^3.$ Recall that $\Kscr$ consists of the accumulation set of points in the (sub)sequence $\{\Kscr_n\}_{n\geq n_0};$ see \cite{Sc}. Denote by $\Sscr=\partial\Kscr$ and let us check that $\Sscr$ is the surface we are looking for.

Denote by $D_{n,j}=X_n(B(p_j,1/n))$ for all $n,$ and by $\overline{S}$ and $\overline{D_j}$ the accumulation set of the (sub)sequence $\{S_n\}_{n\geq n_0}$ and $\{D_{n,j}\}_{n\geq n_0},$ respectively, for all $j.$ Obviously $\overline{S}$ and $\overline{D_j}$ are subsets of $\Sscr.$

From \eqref{eq:areaAnillo} it trivially follows that 
\begin{equation}\label{eq:SUCj}
\Sscr=\overline{S}\cup(\cup_{j=1}^\mgot\overline{D_j});
\end{equation}
otherwise there would exist a subset in $\Sscr$ with positive area, which is the  limit of the sequence $\{X_n(A(p_j,1/n))\}_{n\geq n_0},$ which is impossible. Note also that $S\cap D_j=\emptyset,$ where $S$ and $D_j$ denote the interior of $\overline{S}$ and $\overline{D}_j$ in $\Sscr,$ respectively, for all $j.$ In particular, $\partial S=\cup_{j=1}^\mgot\partial D_j$ and $\Sscr=S\cup(\partial S)\cup (\cup_{j=1}^\mgot D_j).$

Denote by $\Mscr_n$ (respectively, $\Lscr$) the outer parallel surface (respectively, convex body) to $\Sscr_n$ (respectively, to $\Kscr$) at distance $1,$ by $N_{\Sscr_n}:\Sscr_n\to\s^2$ the outer Gauss map of $\Sscr_n$ for all $n\geq n_0,$ and by $\Mscr=\partial \Lscr.$ First of all, let us show that

\begin{claim}\label{cla:Kpart}
$S$ is a $K$-surface. 
\end{claim}
\begin{proof}
Recall that $S_n=X_n(\Sigma_n)$ and denote by
\[
M_n=S_n+(N_{\Sscr_n})|_{S_n}\subset\Mscr_n
\]
the outer parallel surface to $S_n$ at distance $1.$ We know that $\{M_n\}_{n\geq n_0}$ converges to the open set $M\subset\Mscr$ consisting of the points of $\Mscr$ at outer distance $1$ from $S.$ On the other hand, since $S_n$ is a $K$-surface then $M_n$ is a positively curved $H$-surface. In particular, the norm of the second fundamental form of $M_n$ is bounded from above by $1.$ 



So, by standard compactness arguments in constant mean curvature surface theory (see for instance Proposition 2.3 and Lemma 2.4 in \cite{RST}), 
$M$ is an $H$-surface, proving the claim.
\end{proof}

Observe also that, up to passing to a subsequence, $\{S_n\}_{n\geq n_0}$ converges smoothly to $S$ by the same argument.

Now let us check that
\begin{claim}\label{cla:discos}
$D_j$ is an open disc contained in a plane $\Pi_j$ orthogonal to $p_j,$ with ${\rm Area}(D_j)=a_j,$ for all $j=1,\ldots,\mgot.$
\end{claim}
\begin{proof}
Fix $j$ and denote by 
\[
C_{n,j}=D_{n,j}+\sqrt{\frac{a_j n^2}{\pi}}\, \left((N_{\Sscr_n})|_{D_{n,j}}-p_j\right)
\]
the outer parallel surface to $D_{n,j}$ at distance $\sqrt{a_j n^2/\pi},$ translated by the vector $-\sqrt{a_j n^2/\pi}p_j.$ 
It is easy to check that $\|\sqrt{a_j n^2/\pi}\left((N_{\Sscr_n})|_{D_{n,j}}-p_j\right)\|<2\sqrt{a_j/\pi}$ for large enough $n.$ In particular, up to passing to a subsequence, we assume that $\{\sqrt{a_j n^2/\pi}\left((N_{\Sscr_n})|_{D_{n,j}}-p_j\right)\}_{n\geq n_0}$ converges to a vector $v_j\in\r^3.$ So, $v_j+\overline{D}_j$ is the accumulation set of $\{C_{n,j}\}_{n\geq n_0}.$ On the other hand, since $D_{n,j}$ is an open surface with constant Gauss curvature $\pi/a_j n^2$ (see \eqref{eq:f1}), then $C_{n,j}$ is an open surface with constant mean curvature $\frac1{2n}\sqrt{\pi/a_j}$ and positive Gauss curvature. In particular, the square of the norm of the second fundamental form of $C_{n,j}$ is bounded from above by $\pi/a_j n^2.$ In this case, standard compactness arguments in constant mean curvature surface theory (see again Proposition 2.3 and Lemma 2.4 in \cite{RST}) give that $D_j$ is a planar open disc. Since $N_{\Sscr_n}(C_{n,j})=B(p_j,1/n),$ then $D_j$ is contained in a plane orthogonal to $p_j$ and we are done.
\end{proof}

Denote by $N_S$ and $N_{D_j}$ the outer Gauss map of $S$ and $D_j,$ $j=1,\ldots,\mgot,$ respectively. 

\begin{claim}\label{cla:injective}
$N_S:S\to\s^2$ is a homeomorphism onto $\s^2-\{p_1,\ldots,p_\mgot\}.$
\end{claim}
\begin{proof}
Since $\Kscr$ is convex and $S$ is locally strictly convex, then $T_x S\cap \Kscr=\{x\}$ for all $x\in S,$ hence $N_S$ is injective. Likewise, if there would exist $x\in S$ with $N_S(x)=p_j,$ since $\Kscr$ is convex and $\Pi_j$ is a supporting plane of $\Kscr$ (see Claim \ref{cla:discos}), then $x\in \Pi_j.$ However $S$ is locally strictly convex, hence $\Kscr\cap \Pi_j=\{x\},$ a contradiction; see Claim \ref{cla:discos} again. Thus $N_S(S)\subset \s^2-\{p_1,\ldots,p_\mgot\}.$

To check that $N_S:S\to \s^2-\{p_1,\ldots,p_\mgot\}$ is surjective, let $p\in \s^2-\{p_1,\ldots,p_\mgot\}.$ Since $\{\Sigma_n\}_{n\geq n_0}$ is an exhaustion of $\s^2-\{p_1,\ldots,p_\mgot\},$ then there exist $\epsilon>0$ and $n_1\in\n$ such that $B(p,\epsilon)\subset \Sigma_n$ for all $n\geq n_1.$ From {\sf (b$_n$)}, $N_{S_n}(X_n(p))=p$ for all $n\geq n_1.$ On the other hand, up to passing to a subsequence, $\{X_n(p)\in S_n\}_{n\geq n_1}$ converges to a point $x\in S.$ Since the convergence of $\{S_n\}_{n\geq n_0}$ to $S$ is smooth, then $N_S(x)=p,$ proving the claim.
\end{proof}

Notice that Claim \ref{cla:injective} and \eqref{eq:SUCj} give that
the map $N_\Sscr:\Sscr\to\s^2,$
\[
N_\Sscr(p)=
\left\{
\begin{array}{ll}
N_S(p) & p\in S,\\
p_j & p\in \overline{D}_j,\; j\in\{1,\ldots,\mgot\},\\
\end{array}
\right.
\]
is a continuous outer Gauss map of $\Sscr;$ recall that $N_{D_j}=p_j$ for all $j.$ In particular, $\Kscr$ admits a unique supporting plane at every point in $\Sscr,$ hence $\Kscr$ is a smooth convex body and $\Sscr$ is a $\Ccal^1$ surface. Since $\Kscr$ is a convex body, then $S$ is an embedded $K$-surface and $D_j$ is a convex disc for all $j.$ 

On the other hand, the intrinsic conformal structure of the $H$-surface $M=S+N_S$ is clearly biholomorphic to a circular domain in $\overline{\c};$ indeed, just observe that $\overline{M}$ is an embedded constant mean curvature surface with boundary consisting of a finite collection of $\Ccal^1$ Jordan curves.  Thus, so is the extrinsic conformal structure of $S.$

Finally, the uniqueness part of Theorem \ref{th:intro} follows from the one of the generalized Minkowski problem; see Theorem \ref{th:MinGen} and the short discussion after it. 

This completes the proof of Theorem \ref{th:intro}.


\section{Applications}\label{sec:app}

Theorem \ref{th:intro} has direct and interesting applications concerning $H$-surfaces, capillary surfaces in $\r^3,$ harmonic diffeomorphisms between domains of $\s^2,$ and a Hessian equation of Monge-Amp\`ere type.

\subsection{Capillary surfaces}\label{sub:icapi}
Consider a region $\mathcal B$ in $\r^3.$ A capillary surface in $\mathcal B$ is a compact $H$-surface meeting  $\partial \mathcal B$  at a constant angle $\gamma \in [0,\pi]$ along its boundary. Capillary surfaces are stationary surfaces for an energy functional under a volume constraint. More precisely, given a compact surface $\Sigma$ inside $\mathcal B$ such that $\partial \Sigma \subset \partial \mathcal B$ and $\partial \Sigma$ bounds a compact domain $W$ in
$\partial \mathcal B,$ the energy of $\Sigma$ is  by definition the quantity:
$$ \mathcal E(\Sigma)= {\rm Area} (\Sigma) -\cos\gamma \,  {\rm Area (W)}.$$

The stationary surfaces of $\mathcal E$ for variations preserving the enclosed volume are precisely the 
$H$-surfaces which make a constant angle $\gamma$ with  $\partial \mathcal B.$ Here the contact angle is computed from inside the domain enclosed by $\Sigma\cup W.$ Capillary surfaces model liquid drops inside a container in the absence of gravity.  $\Sigma$ represents the free surface of the 
drop and $W$ the region of the container wetted by the drop. A standard reference  on capillary surfaces is the book by Finn \cite{Fi}. 

One can also, more generally, consider compact $H$-surfaces inside a region $\mathcal B$ meeting $\partial \mathcal B$ at a constant angle {\it along each of their boundary components}, the constant possibly varying from one component to the other. Of particular interest is the case when the region $\mathcal B$ is bounded by portions of planes, that is, $\mathcal B$ is a polyhedral region. In the physical interpretation, one allows the bounding faces  of the polyhedral container to be composed of different homogeneous materials.  In the simplest case of a wedge, that is, the region of the space between two intersecting planes, Park \cite{Pa} has shown that if such a surface is  embedded, topologically an annulus  and does not touch the vertex of the wedge, then it has to be part of a round sphere (a partial result had been obtained previously by McCuan \cite{Mc}). Also, Wente \cite{We}, has constructed an immersed, non embedded, annulus of constant mean curvature intersecting orthogonally two parallel planes. 

As an important  consequence of our work, we obtain   a large new family of 
capillary surfaces, with contact angle $\gamma=\pi,$ inside polyhedral regions. More precisely we have the following:

\begin{corollary}\label{co:H}
Condition Theorem \ref{th:intro}-{\sf (i)} is equivalent to either of the following conditions:
\begin{enumerate}[{\sf ({i}.$1$)}]
\item There exists an $H$-surface $S$ with $H=1/2$ such that the intrinsic conformal structure of $S$ is a circular domain $U\subset\overline{\c},$ and the Gauss map of $S$ is a harmonic diffeomorphism $U\to\s^2-\{p_1,\ldots,p_\mgot\}.$
\item There exists a positively curved $H$-surface $S$ with $H=1/2$ such that the intrinsic conformal structure of $S$ is a circular domain $U\subset\overline{\c},$ and the Gauss map of $S$ is a harmonic diffeomorphism $U\to\s^2-\{p_1,\ldots,p_\mgot\}.$
\end{enumerate}

Furthermore, the statement of Theorem \ref{th:intro} holds if one replaces {\sf (i)} by {\sf (i.$2$)}. In particular, if $S$ is as in {\sf (i.$2$)} then it is an embedded $H$-surface which meets tangentially the faces of the polyhedral region determined by the affine planes $\Pi_j,$ $j=1,\ldots,\mgot.$
\end{corollary}
\begin{proof}
Obviously, {\sf (i.$2$)} implies {\sf (i.$1$)}.

Assume now that {\sf (i.$1$)} is true, then as the Gauss map of $S$ is a diffeomorphism, the principal curvatures of $S$ are different from zero. It follows that the parallel surface to $S$ at signed distance $(-1)$ is a regular $K$-surface with $K=1$ which satisfies condition Theorem \ref{th:intro}-{\sf (i)}. 

Finally, if $S$ satisfies Theorem \ref{th:intro}-{\sf (i)} then the outer parallel surface at distance $1$ to $S$ meets the requirements in {\sf (i.$2$)}. 

The last assertion in the corollary straightforwardly follows from the fact that the $H$-surface $S$ in {\sf (i.$2$)} is obtained as the outer parallel surface to a $K$-surface satisfying the statement of Theorem \ref{th:intro}. 
\end{proof}
Obviously the result is valid for any constant $H\neq 0.$ Indeed, up to scaling and changing the orientation, one can always assume $H=1/2.$  We emphasize that the surfaces we obtain have genus zero, meet tangentially  all the faces of the polyhedra  and do not touch their edges.

\subsection{Harmonic diffeomorphisms between domains of $\s^2$}\label{sub:harmonic}

The problem of determining whether there exist harmonic diffeomorphisms between given non-quasiconformally equivalent Riemannian surfaces is an important question with large literature. Heinz \cite{He} proved there is no harmonic diffeomorphism from the unit complex disk $\d$ onto the complex plane $\c,$ with the euclidean metric, and Collin and Rosenberg \cite{CR} showed a harmonic diffeomorphism from $\c$ onto the hyperbolic plane $\h^2,$ disproving a conjecture by Schoen and Yau \cite{SY}; see also \cite{Ma}. As pointed out in Sect. \ref{sec:intro}, circular domains $U$ and harmonic diffeomorphisms $U\to \s^2-\{p_1,\ldots,p_\mgot\}$ were shown in \cite{AS} for any $\{p_1,\ldots,p_\mgot\}\subset\s^2$ with $\mgot\geq 2.$ It is important to note that the existence of such diffeomorphisms is not used in our arguments, hence actually Theorem \ref{th:intro} provides an alternative proof of this fact, under the restriction given by {\sf (ii)}. The authors do not know whether the harmonic diffeomorphisms $U\to \s^2-\{p_1,\ldots,p_\mgot\}$ that follow from Theorem \ref{th:intro} are those shown in \cite{AS}. Moreover, the construction method in the present paper is different from the ones in both \cite{CR} and \cite{AS}, where the harmonic diffeomorphisms were obtained as vertical projection of entire minimal (respectively, maximal) graphs in the Riemannian (respectively, Lorentzian) product manifold $\h^2\times\r$ (respectively, $(\s^2-\{p_1,\ldots,p_\mgot\})\times\r_1$).

The following existence result for harmonic diffeomorphisms between domains of $\s^2$ straightforwardly follows from Theorem \ref{th:intro}.
\begin{corollary}\label{co:harmonic}
Let $\{p_1,\ldots,p_\mgot\}$ be a subset of $\s^2$ with cardinal number $\mgot\in\n$ such that there exist positive numbers $a_1,\ldots,a_\mgot$ with $\sum_{j=1}^\mgot a_jp_j=\vec{0}$ (in particular, $\mgot\geq 2$).

Then there exist a circular domain $U\subset\overline{\c}$ and a harmonic diffeomorphism $U\to\s^2-\{p_1,\ldots,p_\mgot\}.$
\end{corollary}

Recall that the above result remains true for a general subset $\{p_1,\ldots,p_\mgot\}$ of $\s^2$ with $\mgot\geq 2$ \cite{AS}.

\subsection{A Hessian equation}\label{sub:hessian}

There has been considerable research activity in recent years devoted to fully nonlinear,
elliptic second order partial differential equations of the form,
\[
\Fscr[u]:= F\big( \nabla^2 u+ A(\cdot,u,\nabla u)\big)= B(\cdot,u,\nabla u),
\]
in domains $\Omega$ in Euclidean $n$-space, $\r^n,$ as well as their extensions to Riemannian
manifolds. Here the functions $F:\r^n\times\r^n\to\r,$ $A:\Omega\times\r\times\r^n\to\r^n\times\r^n$ and $B:\Omega\times\r\times\r^n\to\r$ are given, the operator $F$ is well-defined classically for functions $u\in \Ccal^2(\Omega),$ and $\nabla^2 u$ and $\nabla u$ denote respectively the Hessian matrix and gradient
vector of $u.$ See for instance \cite{Tr} for a survey in the topic. The following {\em Hessian equations} of Monge-Amp\`ere type are some of the simpler and more studied instances:
\begin{equation}\label{eq:eq}
\Fscr[u]= {\rm det}\big( \nabla^2 u+ c u {\rm I}\big)= f\quad \text{on $\Omega\subset\M^2(c),$}
\end{equation}
where $\M^2(c)\subset\r^3$ is the simply-connected complete Riemannian $2$-manifold with constant curvature $c = 0, 1$ (that is to say, $\M^2(0)=\r^2$ and $\M^2(1)=\s^2$), ${\rm I}$ denotes the identity matrix of the tangent plane $T_p\M^2(c)\subset\r^3,$ for $p\in \M^2(c),$ and $f:\Omega\to\r$ is a positive function; notice that $f$ must be positive for $\Fscr$ to be elliptic.

In case $c=0$ and $f=1,$ one obtains the classical Hessian one equation. In this setting, a celebrated result by J$\ddot{\text{o}}$rgens \cite{J1} states that all solutions to \eqref{eq:eq} globally defined on $\r^2$ are quadratic polynomials, whereas the space of solutions to \eqref{eq:eq} defined on the finitely punctured plane $\r^2-\{p_1,\ldots,p_n\},$ $n\in\n,$ was described by G\'{a}lvez, Mart\'{i}nez, and Mira \cite{GMM}; see \cite{J2} for $n=1.$

On the other hand, in case $c=1$ and $\Omega=\s^2,$ one deals with the classical Minkowski problem; see Sect. \ref{sec:Minkowski}. For a general $\Omega\subset\s^2,$ any solution to \eqref{eq:eq} is the support function of a surface $S$ in $\r^3$ such that the Gauss map of $S$ is a homeomorphism $N_S:S\to\Omega$ and the Gauss curvature function of $S$ is given by $1/(f\circ N_S):S\to\r;$ see Sect. \ref{sec:support} for details. Then, one can check that for $f=1$ and $\Omega=\s^2-\{p_1,\ldots,p_\mgot\},$ $\mgot\in\n,$ any solution $u$ of \eqref{eq:eq} with {\em non-removable singularities} at the points $\{p_1,\ldots,p_\mgot\}$ (that is, $u$ does not $\Ccal^1$-extend across any $p_j$) is the support function of a surface $S$ as those in Theorem \ref{th:intro}-{\sf (i)}. Moreover, Theorem \ref{th:intro} provides a description of the space of solutions to the equation 
\begin{equation}\label{eq:hessiano1}
{\rm det}(\nabla^2 u+ u{\rm I})=1\quad \text{on $\s^2-\{p_1,\ldots,p_\mgot\}.$}
\end{equation}
Note that if $h$ denotes the restriction to $\s^2$ of a linear function on $\r^3$ then $h$ satisfies on $\s^2$ the equation $\nabla^2 h+ h{\rm I}=0.$ So if $u$ is a solution to (\ref {eq:hessiano1}) with non-removable singularities at  the points $\{p_1,\ldots,p_\mgot\}$, then $u+h$ is again a solution. We can thus define an equivalence relation $\sim$ on the set of solutions to  \eqref{eq:hessiano1} as follows:

We say that two solutions $u$ and $v$ of \eqref{eq:hessiano1} are equivalent, and write $u\sim v,$ if $u-v$ is the restriction to $\s^2$ of a linear function on $\r^3.$

\begin{corollary}\label{co:equation}
The space of solutions to the equation \eqref{eq:hessiano1} with non-removable singularities at the points $\{p_1,\ldots,p_\mgot\},$ under the equivalence relation $\sim,$ is in bijection with the set 
\[
\Xi=\big\{(a_1,\ldots,a_\mgot)\in\r^\mgot\;\big|\;a_j> 0\;\forall j=1,\ldots,\mgot,\; \sum_{j=1}^\mgot a_jp_j=\vec{0}\big\}.
\]

Moreover, any solution to \eqref{eq:hessiano1} with non-removable singularities at the points $\{p_1,\ldots,p_\mgot\}$ extends to $\s^2$ as a continuous function.
\end{corollary}
\begin{proof}
Let $u:\s^2-\{p_1,\ldots,p_\mgot\}\to\r$ be a solution to \eqref{eq:hessiano1} with non-removable singularities. Then, the map
\[
X_u:\s^2-\{p_1,\ldots,p_\mgot\}\to\r^3,\quad X_u(p)=\nabla u(p)+u(p)p,
\]
is a $K$-immersion whose Gauss map is the identity map of $\s^2-\{p_1,\ldots,p_\mgot\};$ see \eqref{eq:support-recover}. Since the singularities of $u$ are non-removable, then so are the ones of $X_u,$ hence, by G\'{a}lvez, Hauswirth, and Mira's results \cite{GHM}, the extrinsic conformal structure of $X_u$ is a circular domain $U\subset\overline{\c}.$ Therefore, Theorem \ref{th:intro} applies and there exist planar discs $D_1,\ldots,D_\mgot$ in $\r^3$ such that $\Sscr=X_u(\s^2-\{p_1,\ldots,p_\mgot\})\cup(\cup_{j=1}^\mgot \overline{D}_j)$ is a smooth convex body and $\sum_{j=1}^\mgot a_j(u)p_j=\vec{0},$ where $a_j(u)={\rm Area}(D_j)>0.$ 

 Note that if $u$ and $v$ are solutions to \eqref{eq:hessiano1} with $u\sim v,$ then $X_u$ and $X_v$ differ by a translation. Indeed, if $u-v=\langle \vec c, . \rangle,$ where $\vec c=(c_1,c_2,c_3)\in \r^3,$   then $X_u-X_v= \vec c.$ So the map 
\[
u\mapsto (a_1(u),\ldots,a_\mgot(u)),
\]
is well defined  from the space of solutions to \eqref{eq:hessiano1} with non-removable singularities under $\sim$ into the set $\Xi.$  The uniqueness part of Theorem \ref{th:intro} trivially implies that this map is injective. On the other hand, the surjectivity of the map follows from the fact that {\sf (ii)} implies {\sf (i) }in the theorem.

Finally, the second assertion in the statement of the corollary can be derived from the proof of Claim \ref{cla:C1}. 
\end{proof}



\begin{thebibliography}{12}

\bibitem{AS} A. Alarc\'{o}n and R. Souam, {\em Harmonic  diffeomorphisms between domains in the Euclidean $2$-sphere.} Comment. Math. Helv., in press.

\bibitem{Br} H. Busemann, {\em Convex surfaces.} Interscience Tracts in Pure and Applied Mathematics, no. 6 Interscience Publishers, Inc., New York; Interscience Publishers Ltd., London, 1958.

\bibitem{CY} S.Y. Cheng and S.T. Yau, {\em On the regularity of the solution of the $n$-dimensional Minkowski problem.} Comm. Pure Appl. Math. {\bf 29} (1976), no. 5, 495--516.

\bibitem{CR} P. Collin and H. Rosenberg, {\em Construction of harmonic diffeomorphisms and minimal graphs.} Ann. of Math. (2) {\bf 172} (2010), 1879--1906.

\bibitem{Fi} R. Finn, {\em Equilibrium capillary surfaces. }  Grundlehren der Mathematischen Wissenschaften, 284. Springer-Verlag, New York, 1986. 

\bibitem{GHM} J.A. G\'{a}lvez, L. Hauswirth, and P. Mira, {\em Surfaces of constant curvature in $\r^3$ with isolated singularities.} Preprint (arXiv:1007.2523).

\bibitem{GM} J.A. G\'{a}lvez and A. Mart\'{i}nez, {\em The Gauss map and second fundamental form of surfaces in $\r^3.$} Geom. Dedicata {\bf 81} (2000), no. 1-3, 181-–192.

\bibitem{GMM} J.A. G\'{a}lvez, A. Mart\'{i}nez, and P. Mira, {\em The space of solutions to the Hessian one equation in the finitely punctured plane.} J. Math. Pures Appl. (9) {\bf 84} (2005), no. 12, 1744-–1757.

\bibitem{GMi} J.A. G\'{a}lvez and P. Mira, {\em
Embedded isolated singularities of flat surfaces in hyperbolic 3-space.}
Calc. Var. Partial Differential Equations {\bf 24} (2005), no. 2, 239-–260. 

\bibitem{He} E. Heinz, {\em $\ddot{\text{U}}$ber die L$\ddot{\text{o}}$sungen der Minimalfl$\ddot{\text{a}}$chengleichung.} Nachr. Akad.
Wiss. G$\ddot{\text{o}}$ttingen. Math.-Phys. Kl. (1952), 51--56.

\bibitem{J1} K. J$\ddot{\text{o}}$rgens, {\em $\ddot{\text{U}}$ber die L$\ddot{\text{o}}$sungen der Differentialgleichung $rt-s^2=1.$} Math. Ann.
{\bf 127} (1954), 130-–134.

\bibitem{J2} K. J$\ddot{\text{o}}$rgens, {\em Harmonische Abbildungen und die Differentialgleichung $rt-s^2 = 1.$}
Math. Ann. {\bf 129} (1955), 330-–344.

\bibitem{L} H. Lewy, {\em On differential geometry in the large. I. Minkowski's problem.}
Trans. Amer. Math. Soc. {\bf 43} (1938), no. 2, 258–-270. 

\bibitem{LSZ} A.M. Li, U. Simon, and G.S. Zhao, {\em Global affine differential geometry of hypersurfaces.}
de Gruyter Expositions in Mathematics, 11. Walter de Gruyter $\&$ Co., Berlin, 1993.

 \bibitem{Mc} J. McCuan, {\em  Symmetry via spherical reflection and spanning drops in a wedge.} Pacific J. Math. {\bf 180} (1997), no. 2, 291--323.
 
\bibitem{Ma} V. Markovic, {\em Harmonic diffeomorphisms and conformal distortion of Riemann surfaces.} Comm. Anal. Geom. {\bf 10} (2002), no. 4, 847-–876.

\bibitem{NR} B. Nelli and H. Rosenberg, {\em Some remarks on positive scalar and Gauss-Kronecker curvature hypersurfaces of $\r^{n+1}$ and $\mathbb{H}^{n+1}.$} Ann. Inst. Fourier (Grenoble) {\bf 47} (1997), no. 4, 1209-–1218.

\bibitem{N} L. Nirenberg, {\em The Weyl and Minkowski problems in differential geometry in the large. }
Comm. Pure Appl. Math. {\bf 6} (1953),  337--394.

\bibitem{Pa} S.-h. Park, {\em 
Every ring type spanner in a wedge is spherical. } 
Math. Ann. {\bf 332} (2005), no. 3, 475--482.

\bibitem{P} A.V. Pogorelov,  {\em  Regularity of a convex surface with given Gaussian curvature. }
Mat. Sbornik N.S. {\bf 31} (73), (1952). 88--103.

\bibitem{RST} H. Rosenberg, R. Souam, and E. Toubiana, {\em General curvature estimates for stable $H$-surfaces in $3$-manifolds and applications.} J. Differential Geom. {\bf 84}  (2010), no. 3, 623--648.
\bibitem{Ru} E.A. Ruh, {\em Asymptotic behaviour of non-parametric minimal hypersurfaces.}
J. Differential Geometry {\bf 4} (1970), 509-–513.

\bibitem{Sc} R. Schneider, {\em Convex bodies: the Brunn-Minkowski theory.} Encyclopedia of Mathematics and its Applications, 44. Cambridge University Press, Cambridge, 1993.

\bibitem{SY} R. Schoen and S.T. Yau, {\em Lectures on Harmonic Maps.}
Conference Proceedings and Lecture Notes in Geometry and Topology,
II. International Press, Cambridge, MA, 1997.

\bibitem{Sp} M. Spivak, {\em A comprehensive introduction to differential geometry. Vol. III.} Second edition. Publish or Perish, Inc., Wilmington, Del., 1979.

\bibitem{Tr} N. Trudinger, {\em Recent developments in elliptic partial differential equations of Monge-Amp\`ere type.} International Congress of Mathematicians. Vol. III, 291–-301, Eur. Math. Soc., Z$\ddot{\text{u}}$rich, 2006.

\bibitem{We} H.C. Wente, {\em Tubular capillary surfaces in a convex body.}  Advances in geometric analysis and continuum mechanics (Stanford, CA, 1993), 288--298, Int. Press, Cambridge, MA, 1995.

\end{thebibliography}
\end{document}